\documentclass[a4paper,12pt]{article}

\usepackage{amsthm}
\usepackage{mathrsfs}
\usepackage{palatino}
\usepackage{amsfonts,amsmath,amssymb,amscd}

\usepackage{epsfig}

\usepackage{euscript}
\usepackage{amscd}
\usepackage{amsthm}
\usepackage{theoremref}
\usepackage{tikz-cd}
\DeclareMathAlphabet{\mathpzc}{OT1}{pzc}{m}{it}

\newtheorem{thm}{Theorem}[section]

\newtheorem{lem}[thm]{Lemma}
\newtheorem{prop}[thm]{Proposition} 

\newtheorem{rem}[thm]{Remark}

\newcommand{\m}{\mathpzc{m}}
\newcommand{\p}{\mathpzc{p}}

\newcommand{\bQ}{\mathbb Q}

\newcommand{\A}{\mathbb A}

\newcommand{\Spec}{\operatorname{Spec}}

\newcommand{\td}{\operatorname{tr.deg}}

\title{On the triviality of an $\A^2$-fibration over a DVR}
\author{Parnashree Ghosh$^*$ and Neena Gupta$^{**}$\\
	{\small{\it  Theoretical Statistics and Mathematics  Unit, Indian Statistical Institute,}}\\ 
	{\small{\it 203 B.T.Road, Kolkata-700108, India}}\\
	{\small{\it e-mail : $^*$parnashree$\_$r@isical.ac.in, ghoshparnashree@gmail.com}}\\
	{\small{\it e-mail : $^{**}$neenag@isical.ac.in, rnanina@gmail.com}}
}

\begin{document}
	
	\date{}
	\maketitle
	\abstract{In this paper we show that any $\mathbb{A}^2$-fibration over a discrete valuation ring which is also an $\mathbb{A}^2$-form is necessarily a polynomial ring. Further we show that separable $\mathbb{A}^2$-forms over PIDs are trivial.
	
	\smallskip
	
		\smallskip
		

	\section{Introduction}
 Throughout this paper by a `ring' we will mean a commutative ring with identity and by an `algebra' we will mean a commutative algebra.

 Let $R$ be a ring and $A$ an $R$-algebra. 
 The notation $R^{[n]}$ denotes a polynomial ring in $n$ variables over $R$ and for $\p \in \Spec(R)$,  $\kappa(\p)$ denotes the residue field $\frac{R_{\p}}{\p R_{\p}}$,
  and $A_{\p}$ denotes the ring $S^{-1}A$, where $S:=R \setminus \p$.
  As in \cite{S}, we shall call a finitely generated flat $R$-algebra $A$ to be an $\mathbb{A}^n$-fibration if $A \otimes_R \kappa(\p)=\kappa(\p)^{[n]}$ for every $\p \in \Spec(R)$. 

Now let $(R, \pi)$ be a discrete valuation ring (DVR) with residue field $\kappa:= R/\pi R$ and field of fractions $K:= R[1/\pi]$. Suppose  that $A$ is an $R$-algebra such that 
\begin{enumerate}
\item[\rm(i)] $A$ is $R$ flat.
\item[\rm(ii)] $A\otimes_R K (=A[1/\pi])=K^{[2]}$ and
\item[\rm(iii)] $A\otimes_R \kappa (=A/\pi A) = \kappa^{[2]}$.
\end{enumerate}
By a result of T. Asanuma (\cite[Theorem 3.1]{A}), it follows that $A$ is a finitely generated $R$-algebra, and hence $A$ is an $\A^2$-fibration. 
A natural question is whether $A=R^{[2]}$. This question was answered affirmatively by A. Sathaye in \cite{S} under the assumption that $R$ contains the field of rational numbers $\bQ$.

On  the other hand, T. Asanuma constructed an example of an algebra $A$ over a DVR $R$ which does not contain $\bQ$, satisfies (i), (ii) and (iii) but is not $R^{[2]}$ \cite[Theorem 5.1]{A}. 
Therefore it is natural to ask the following question:

\medskip
\noindent
{\bf Question 1.} Let $(R,\pi)$ be a DVR not containing $\bQ$  and $A$ an $\mathbb{A}^2$-fibration over $R$. What additional condition will ensure that $A=R^{[2]}$?

\medskip
 In section 2 of this paper we prove the following result (\thref{trivialfib}):

\medskip
\noindent
{\bf Theorem A.}
Let $(R,\pi)$ be a DVR and $A$ an $\mathbb{A}^2$-fibration over $R$. 
Let $K:=R[\frac{1}{\pi}]$, $L$ be a finite field extension of $K$ and $\tilde{R}$ be the integral closure of $R$ in $L$.
If $A \otimes_R \tilde{R} = \tilde{R}^{[2]}$, then $A=R^{[2]}$.

\medskip

We have given another important criterion for an $\mathbb{A}^2$-fibration over a DVR to be trivial.
Before stating the result, we discuss the concept of $\mathbb{A}^n$-form over a ring. 

Let $B$ be a ring containing a field $k$ and $A$ be a $B$-algebra. 
 We say that $A$ is an {\it $\mathbb{A}^n$-form over $B$} with respect to $k$, if $A\otimes_k \overline{k}= (B\otimes_k \overline{k})^{[n]}$, where $\overline{k}$ is an algebraic closure of $k$.
 If, in addition, $\overline{k}$ is a separable field extension of $k$, then we call $A$ a {\it separable $\mathbb{A}^n$-form} over $B$ with respect to $k$.
    
 When $B=k$, then it is well known that separable $\A^1$-forms over $k$ are necessarily polynomial rings (see \cite[Lemma 5]{DA1}).
 Further, T. Kambayashi has shown the following (\cite[Theorem 3]{K}): 
 
 \medskip
 \noindent
 {\bf Theorem B.}
 Let $k$ be a field and $A$ be a separable $\A^2$-form over $k$.  Then $A=k^{[2]}$.
 
 \medskip
 
However, there are plenty of examples of non-trivial $\A^1$-forms (i.e., $\A^1$-forms which are not polynomial rings) over non-perfect fields (\cite{insA1}).
  
  In \cite{DA1}, A. K. Dutta proved the following result:
  
  \medskip
  \noindent
  {\bf Theorem C.}
  Let $k$ be a field of characteristic zero, $R$ be a PID containing $k$ and $A$ be an $\A^2$-form over $R$ with respect to $k$. Then $A=R^{[2]}$. 
  
  \medskip
  However if $k$ is a non-perfect field, then using non-trivial $\mathbb{A}^1$-forms over $k$, we can construct examples of non-trivial $\mathbb{A}^2$-forms over PIDs.
  Therefore we ask the following question:
  
  \medskip
  \noindent
  {\bf Question 2.}
  Let $R$ be a PID containing a field $k$ of arbitrary characteristic and $A$ be an $\mathbb{A}^2$-form over $R$ with respect to $k$. What additional condition will ensure $A=R^{[2]}$?

\medskip
 In section 2 of this paper we prove the following result (\thref{corform}):
 
 \medskip
 \noindent
 {\bf Theorem D.}
 	Let $k$ be a field and $R$ be a PID which contains $k$.
 	 If $A$ is an $\mathbb{A}^{2}$-fibration over $R$ as well as an $\mathbb{A}^2$-form over $R$ with respect to $k$, then $A=R^{[2]}$.

\medskip
	
	Note that the above result gives an answer to both Questions 1 and 2. Furthermore, as an application of Theorem D, we prove the following (\thref{perfect}):
	
	\medskip
	\noindent
	{\bf Theorem E.}
	Let $k$ be a field and $R$ be a PID containing $k$ and $A$ be a separable $\mathbb{A}^2$-form over $R$ with respect to $k$. Then $A=R^{[2]}$. 
	
	\medskip
	
	 In particular, Theorem E extends Theorems B and C to any PID containing an arbitrary field.

	\section{Main Theorems}

We first state some results which will be used in this paper. The following is a well known result (\cite[2.11]{AEH}).

\begin{thm}\thlabel{aeh}
	Let $k$ be a field and $R$ be a normal domain such that $k \subset R \subset k^{[n]}$. If $\td_{k} R=1$, then $R = k^{[1]}$.
\end{thm}

The following theorem was proved by Bass, Connell and Wright in \cite{bcw}, and independently by Suslin in \cite{su}.

\begin{thm}\thlabel{bcw}
	Let $A$ be a ring and $B$ be a finitely presented $A$-algebra. Suppose that there exists $n \in \mathbb{N}$ such that $B_{\m}={A_{\m}}^{[n]}$ for every maximal ideal $\m$ of $A$.
	 Then $B \cong Sym_{A}(M)$ for some finitely generated projective $A$-module $M$ of rank $n$.
\end{thm}

The next result is due to Asanuma (\cite[Corollary 4.9]{A}). 
 The following version can be found in \cite[Theorem 2.7]{BDlin}.

\begin{thm}\thlabel{asa}
Let $(B,p)$ be a DVR, $\kappa=B/p B$ and $L=B\left[\frac{1}{p}\right]$. Let $X_0,Y_0 \in B[X,Y] (=B^{[2]})$ and $\overline{X_0}$, $\overline{Y_0}$ denote the images of $X_0,Y_0$ in $\kappa[X,Y]$. Suppose that $\overline{X_0},\overline{Y_0} \notin \kappa$ and $L[X_0,Y_0]=L[X,Y]$. Then either $\deg \overline{X_0} \mid \deg \overline{Y_0}$ or 	$\deg \overline{Y_0} \mid \deg \overline{X_0}$, where $\deg$ means the total degree in $X,Y$. 
\end{thm}
%

We note below an elementary lemma.
\begin{lem}\thlabel{el}
Let $C \subseteq D$ be integral domains and $t \in C$ be such that $D[1/t]=C[1/t]$
and $tD \cap C= tC$. Then $C=D$. 
\end{lem}

We now prove Theorem A.
\begin{thm}\thlabel{trivialfib}
	Let $(R,\pi)$ be a DVR and $A$ an $\mathbb{A}^2$-fibration over $R$. 
	 Let $K:=R\left[\frac{1}{\pi}\right]$, $\kappa=\frac{R}{\pi R}$, $L$ a finite field extension of $K$ and $\tilde{R}$ be the integral closure of $R$ in $L$.
	 If $A \otimes_R \tilde{R} = \tilde{R}^{[2]}$, then $A=R^{[2]}$.
\end{thm}
\begin{proof}
		For any $a \in A$, let $\overline{a}$ denote the image of $a$ in $\frac{A}{\pi A}$.
	Since $A$ is an $\mathbb{A}^2$-fibration over $R$, there exist $X_1,Y_1,X_0,Y_0 \in A$ such that 
	\begin{equation}\label{K}
		A\left[\frac{1}{\pi}\right] = K[X_1,Y_1] 
	\end{equation}
	and 
	\begin{equation}\label{kp}
		\frac{A}{\pi A} = (R/\pi R)[\overline{X_0},\overline{Y_0}]=\kappa[\overline{X_0},\overline{Y_0}].
	\end{equation}
	Let $A=R[t_1,\ldots,t_m]$. By \eqref{K}, $t_i=\frac{q_i(X_1,Y_1)}{\pi^{n_i}}$ for every $i$, $1 \leqslant i \leqslant m$ and $q_i \in R^{[2]}$ with $q_i \notin \pi R[X_1, Y_1]$.
	Let $N_{X_1,Y_1} :=\sum_{i=1}^{m}\max\{n_i,0\}$. 
	Note that if $N_{X_1,Y_1}=0$, then $A=R[X_1,Y_1]$. 
	
	We show that for every pair of elements $X,Y \in A$ 
	such that $A \otimes_{R} K=K[X,Y]$ with $N_{X,Y}>0$, there exists a modified pair
	$\tilde{X}, \tilde{Y} \in A$ such that $A \otimes_{R} K=K[\tilde{X},\tilde{Y}]$ with $N_{\tilde{X}, \tilde{Y}} < N_{X,Y}$. Then by induction on $N_{X,Y}$, we see that $A=R^{[2]}$.

	If $\td_{\kappa} \kappa [\overline{X_1}, \overline{Y_1}]=2$, then 
	$\frac{R[X_1,Y_1]}{\pi R[X_1,Y_1]} \hookrightarrow \frac{A}{\pi A}$. Now since $A[\frac{1}{\pi}]=R[\frac{1}{\pi}][X_1,Y_1]$, it follows that $A=R[X_1,Y_1]$ by \thref{el}. Clearly in this case $N_{X_1, Y_1}=0$.
	
	We now show that if $\td_{\kappa} \kappa[\overline{X_1}, \overline{Y_1}]$ is zero or one, then there exist $X_2,Y_2 \in A$ with $N_{X_2,Y_2}< N_{X_1,Y_1}$.
    
    If $\td_{\kappa} \kappa[\overline{X_1}, \overline{Y_1}]=0$, then as $\kappa$ is algebraically closed in $\frac{A}{\pi A}=\kappa^{[2]}$, we have $X_1-\lambda, Y_1-\mu \in \pi A$, for some $\lambda, \mu \in R$.  Therefore, for $X_2:= \frac{X_1-\lambda}{\pi}$ and $Y_2:= \frac{Y_1-\mu}{\pi}$, we have $N_{X_2,Y_2}<N_{X_1,Y_1}$, as in \cite[Theorem 1]{S}.
	
	If $\td_{\kappa} \kappa[\overline{X_1}, \overline{Y_1}]=1$, then  by \thref{aeh}, the normalization of $\kappa[\overline{X_1}, \overline{Y_1}]$ in $\frac{A}{\pi A}=\kappa[\overline{X_0}, \overline{Y_0}]$ is $\kappa^{[1]}$. 
	Hence
	\begin{equation}\label{w1}
			\kappa[\overline{X_1}, \overline{Y_1}] \hookrightarrow \kappa[W] \hookrightarrow \frac{A}{\pi A}=\kappa[\overline{X_0}, \overline{Y_0}]
	\end{equation}
    for some $W \in \frac{A}{\pi A}$.
	Let $d_{X_1}:= \deg_{W} \overline{X_1}$ and $d_{Y_1}:= \deg_{W} \overline{Y_1}$.  Now suppose the following holds:
	\begin{equation}\label{fact}
		\text{~Either~}	\min(d_{X_1},d_{Y_1})=0, \text{~or~} d_{X_1},d_{Y_1}>0 \text{~with~}
		d_{X_1} \mid d_{Y_1} \text{~or~} d_{Y_1} \mid d_{X_1}.
	\end{equation}
	 Then, by \cite[Theorem 1]{S}, there exist $X_2,Y_2 \in A$ such that $A \otimes_{R} K=K[X_2,Y_2]$ with $N_{X_2,Y_2}<N_{X_1,Y_1}$. 
	 We now show that \eqref{fact} holds.
	  
	  From the Krull-Akizuki Theorem (\cite[Theorem 11.7]{matr}), it follows that $\tilde{R}$ is a Dedekind domain. 
	  Let  $\p_i$ be a maximal ideal of $\tilde{R}$.
	 Let $\tilde{A}_{\p_i}:=A \otimes_R \tilde{R}_{\p_i}=\tilde{R}_{\p_i}[u,v] (=\tilde{R}_{\p_i}^{[2]})$ for some $u,v \in A \otimes_R \tilde{R}_{\p_i}$. Let $\tilde{\kappa}_i:= \frac{\tilde{R}_{\p_i}}{\p_i \tilde{R}_{\p_i}}$. 
	 For any $a \in \tilde{A}_{\p_i}$, let $\hat{a}$ denotes its image in $\frac{\tilde{A}_{\p_i}}{\p_i \tilde{A}_{\p_i}}$. Note that $\kappa \subseteq \tilde{\kappa}_i$ and 
	 $$
\frac{\tilde{A}_{\p_i}}{\p_i \tilde{A}_{\p_i}}
= A \otimes_R \tilde{\kappa}_i 
= A \otimes_R \kappa \otimes_{\kappa} \tilde{\kappa}_i
= \frac{A}{\pi A} \otimes_{\kappa} \tilde{\kappa}_i
=\kappa[\overline{X_0}, \overline{Y_0}] \otimes_{\kappa} \tilde{\kappa}_i
=\tilde{\kappa}_i[\hat{X_0}, \hat{Y_0}].
	 $$ 
	 Therefore, using $\tilde{A}_{\p_i}=\tilde{R}_{\p_i}[u,v]$, we have 
	 \begin{equation}\label{u}
\frac{\tilde{A}_{\p_i}}{\p_i \tilde{A}_{\p_i}}
=\tilde{\kappa}_i[\hat{X_0}, \hat{Y_0}]
=\tilde{\kappa}_i[\hat{u}, \hat{v}].
\end{equation}	   
     Now, from \eqref{w1}, we have
	 $$
	 \tilde{\kappa}_i[\hat{X_1},\hat{Y_1}] \hookrightarrow \tilde{\kappa}_i[\hat{W}] \hookrightarrow \tilde{\kappa}_i[\hat{X_0},\hat{Y_0}]=\tilde{\kappa}_i[\hat{u},\hat{v}],
	 $$
	 and hence 
	 \begin{equation}\label{deg}
	 	 \deg_{(\hat{u},\hat{v})} \hat{X_1}=\deg_{(\hat{u},\hat{v})} \hat{W}.d_{X_1} \text{~and~}\deg_{(\hat{u},\hat{v})} \hat{Y_1}= \deg_{(\hat{u},\hat{v})} \hat{W}.d_{Y_1}. 
	 \end{equation}
     We note that $\hat{W}$ is transcendental over $\tilde{\kappa}_i$, 
     $\tilde{R}_{\p_i}\left[ \frac{1}{p_i}\right]=L$ and
     $$
     \tilde{A}_{\p_i}\left[\frac{1}{p_i}\right]
     =A \otimes_R \tilde{R}_{\p_i} \otimes_{\tilde{R}_{\p_i}} 
     \tilde{R}_{\p_i}\left[\frac{1}{p_i}\right]= A \otimes_R L= (A \otimes_R K) \otimes_{K} L=K[X_1,Y_1]\otimes_K L=L[X_1,Y_1].
     $$
Therefore, using $\tilde{A}_{\p_i}=\tilde{R}_{\p_i}[u,v]$, we have
\begin{equation}\label{Rpi}
    \tilde{A}_{\p_i}\left[\frac{1}{p_i}\right]=L[X_1,Y_1]=L[u,v].
\end{equation}
	 If $\hat{X_1}$ or $\hat{Y_1}$ is in $\tilde{\kappa}_i$, then from \eqref{deg}, either $d_{X_1}=0$ or $d_{Y_1}=0$.
	  Now suppose that $\min(d_{X_1},d_{Y_1}) \neq 0$.
	  Then $\hat{X_1}, \hat{Y_1} \notin \tilde{\kappa}_i$, and hence, by \eqref{Rpi} and \thref{asa},  
	  $$
	 \text{either~} \deg_{(\hat{u},\hat{v})} \hat{X_1}\mid \deg_{(\hat{u},\hat{v})} \hat{Y_1}
	  \text{~or~} \deg_{(\hat{u},\hat{v})} \hat{Y_1}\mid \deg_{(\hat{u},\hat{v})} \hat{X_1}.
	  $$ 
	  Now using \eqref{deg}, it follows that either $d_{X_1} \mid d_{Y_1}$ or $d_{Y_1} \mid d_{X_1}$ as $\deg_{(\hat{u},\hat{v})} \hat{W} \neq 0$. Hence the desired result follows. 
	\end{proof}

 The following remark will be useful for proving the results in the rest of this section.
 
 \begin{rem}\thlabel{r1}
 	\em{Let $R$ be a ring containg a field $k$ and $A$ be an $R$-algebra. Suppose that $A\otimes_k L =(R \otimes_k L)^{[n]}$ for some field extension $L$ of $k$. Then, by \cite[Lemma 1]{DA1}, it follows that $A$ is a finitely generated $R$-algebra, and hence by \cite[Lemma 3]{DA1}, we have a finitely generated field extension $\tilde{L}$ of $k$ contained in $L$ such that $A\otimes_k \tilde{L} =(R \otimes_k \tilde{L})^{[n]}$.}
 \end{rem}
 
  The next result is the main step to prove Theorem D.
 
\begin{prop}\thlabel{pform2}
	Let $k$ be a field, $L$ a field extension of $k$, $R$ a DVR which contains $k$ and $A$ an $\mathbb{A}^{2}$-fibration over $R$. If $A \otimes_k L = (R \otimes_k L)^{[2]}$, then $A=R^{[2]}$.
\end{prop}
\begin{proof}
	By \thref{r1}, we can assume that
	$L$ is a finitely generated field extension of $k$. 
	
	\smallskip
	\noindent
	{\it Case} 1: Suppose $L$ is algebraic over $k$.
	 We have
	\begin{align}\label{l1}
		A \otimes_R (R \otimes_k L) = (R \otimes_k L)^{[2]}. 
	\end{align}
	Since $R \subseteq R \otimes_k L$ is an integral extension, there exists a $\p \in \Spec(R \otimes_k L)$ such that $\p \, \cap\, R= (0)$. Therefore $R \subseteq \frac{R \otimes_k L}{\p}$ is a module-finite extension of integral domains.
	Let $K_1$ and $K_2$ denote the field of fractions of $R$ and $ \frac{R \otimes_k L}{\p}$, respectively.
	Since $ \frac{R \otimes_k L}{\p}$ is a finite $R$-module, $ K_2$ is a finite field extension of $K_1$. Now suppose that $\tilde{R}$ be the integral closure of $R$ in $K_2$. Since $\frac{R \otimes_k L}{\p} \subseteq \tilde{R}$, from \eqref{l1}, it follows that $A \otimes_R \tilde{R} = \tilde{R}^{[2]}$. Now, by \thref{trivialfib}, we get the desired result.

	\smallskip
	\noindent
	{\it Case} 2: 
	 Suppose $L$ is not algebraic over $k$.
	 Let $x_1,\ldots,x_n \in L$ be algebraically independent elements over $k$ such that $\tilde{k}:=k(x_1,\ldots,x_n)$ and $L=\tilde{k}[\alpha_1,\ldots,\alpha_n]$ for some algebraic elements $\alpha_1,\ldots,\alpha_n$ over $\tilde{k}$. 
	We now have 
	\begin{align*}
		(A \otimes_k k(x_1,\ldots,x_n)[\alpha_1,\ldots,\alpha_n]) =(R \otimes_k k(x_1,\ldots,x_n)[\alpha_1,\ldots,\alpha_n])^{[2]}.
	\end{align*}
 Since $A$ is a finitely generated $R$-algebra, there exists $f \in k[x_1,\ldots,x_n] (=k^{[n]})$ such that
 	\begin{align}\label{a1}
 	\left(A \otimes_k k\left[x_1,\ldots,x_n,\frac{1}{f},\alpha_1,\ldots,\alpha_n\right]\right) =\left(R \otimes_k k\left[x_1,\ldots,x_n, \frac{1}{f},\alpha_1,\ldots,\alpha_n\right]\right)^{[2]}.
 \end{align} 
Let $\m$ be a maximal ideal of $E:= k\left[x_1,\ldots,x_n,\frac{1}{f},\alpha_1,\ldots,\alpha_n\right]$and $\hat{k}:= \frac{E}{\m}$. Then 
from $\eqref{a1}$, 
\begin{equation}\label{a2}
	A \otimes_k \hat{k} =(R \otimes_k \hat{k})^{[2]}.
\end{equation}
 Note that $\hat{k}$ is algebraic over $k$.
Therefore, the result follows by Case 1.
\end{proof}
We now deduce Theorem D.

\begin{thm}\thlabel{corform}
	Let $k$ be a field and $R$ be a Dedekind domain containing $k$. If $A$ is an $\A^2$-fibration over $R$ which is also an $\A^2$-form over $R$ with respect to $k$, then $A= (\rm{Sym}_R(I))^{[1]}$, where $I$ is an invertible ideal of $R$. In particular, if $R$ is a PID, then $A=R^{[2]}$.
\end{thm}

\begin{proof}
	Let $\p \in \Spec(R)$. Since $A$ is an $\A^2$-fibration and $\A^2$-form over $R$ with respect to $k$, it follows that $A_{\p}:=A \otimes_R R_{\p}$ is an $\A^2$-fibration over $R_{\p}$ and $\A^2$-form over $R_{\p}$ with respect to $k$. Since $R_{\p}$ is a DVR, by \thref{pform2}, we have $A_{\p}=R_{\p}^{[2]}$. 
	Therefore, by \thref{bcw}, $A= \rm{Sym}_R(Q)$ for some finitely generated projective $R$-module $Q$ of rank 2.
	Now, $Q=I \oplus R$ where $I$ is an invertible ideal of $R$ (cf. \cite[Theorem 7.1.8]{rao}). 
	Therefore, $A=(\rm{Sym}_R(I))^{[1]}$. 
	
	If $R$ is PID, then $Q$ is free, and hence $A=R^{[2]}$. 
\end{proof}


We now prove Theorem E. Recall that a field extension $L$ over a field $k$ is said to be separable if, for every field extension $k^{\prime}$ of $k$, $L \otimes_k k^{\prime}$ is a reduced ring. 
  \begin{thm}\thlabel{perfect}
  	Let $k$ be a field, $R$ a Dedekind domain containing $k$. If there exists a separable field extension $L$ of $k$ such that $A \otimes_k L= (R \otimes_k L)^{[2]}$, then $A= (\rm{Sym}_R(I))^{[1]}$, where $I$ is an invertible ideal of $R$. In particular, if $R$ is a PID, then $A=R^{[2]}$.
  \end{thm}
  	\begin{proof}
  			By \thref{r1} and the fact that any subfield of a separable extension is separable, we can assume that $L$ is a finitely generated separable field extension of $k$. 
  			   
     \smallskip
     \noindent
     {\it Case} 1:
     We first assume that $L$ is an algebraic extension of $k$. 
  	Since
  		\begin{equation}\label{b1}
  		A \otimes_k L= (R \otimes_k L)^{[2]},
  	\end{equation}
  we have
   	   \begin{equation}\label{b}
   	   	 \frac{A_{\p}}{\p A_{\p}}\otimes_{\kappa(\p)} (\kappa(\p) \otimes_k L)= (\kappa(\p) \otimes_k L)^{[2]},
   	   \end{equation}
      for every $\p \in \Spec(R)$.
  	 	As $L$ is a finite separable field extension of $k$, $\kappa(\p) \otimes_k L= \prod_{i=1}^{n} L_i$, where $L_i$ is a finite separable field extension of $\kappa(\p)$. Therefore, from \eqref{b}, 
  	 	$$
  	 	\frac{A_{\p}}{\p A_{\p}}\otimes_{\kappa(\p)} L_i= L_i^{[2]},
  	 	$$ 
  	 	 and hence, by Kambayashi's Theorem (Theorem B), it follows that $\frac{A_{\p}}{\p A_{\p}}=\kappa(\p)^{[2]}$. 
  	 	As $A$ is an $\A^2$-form over $R$ and hence flat over $R$,
  	 	  $A$ is an $\mathbb{A}^2$-fibration over $R$. 
         Now the assertion follows by \thref{corform}. 
         
         \smallskip
         \noindent
         {\it Case} 2:
         Now suppose $L$ is not algebraic over $k$. Therefore,
         $L=k(x_1,\ldots,x_n)[\alpha]$ for some $x_1,\ldots,x_n,\alpha \in L$ such that $\{x_1,\ldots,x_n\}$ is a separating transcendence basis for $L$ over $k$, and $\alpha$ is algebraic over $k(x_1,\ldots,x_n)$ (cf. \cite[Theorem 26.2]{matr}).
         Since $A$ is a finitely generated $R$-algebra, from \eqref{b1}, it is clear that there exists $f \in k[x_1,\ldots,x_n] (=k^{[n]})$ such that
         \begin{equation}\label{b2}
         	A \otimes_k k\left[ x_1,\ldots,x_n, \frac{1}{f}, \alpha \right] = \left( R \otimes_k k\left[ x_1,\ldots,x_n, \frac{1}{f}, \alpha \right]  \right)^{[2]}.
         \end{equation}
         Set $E:=k\left[ x_1,\ldots,x_n, \frac{1}{f}, \alpha \right]$.
      If $k$ is a finite field, then for any maximal ideal $\m$ of $E$, $\tilde{k}:=\frac{E}{\m}$ is a finite separable field extension of $k$. 
         Now, from \eqref{b2}, $A\otimes_k \tilde{k} =(R \otimes_k \tilde{k})^{[2]}$, and hence the result follows by Case 1.
         
         Now suppose that $k$ is an infinite field. Set $\hat{k}:=k(x_1,\ldots,x_n)$. 
         Let $P= \sum_{i=0}^{n} g_iY^{i}\in \hat{k}[Y](=\hat{k}^{[1]})$, where $g_n=1$, $g_i \in \hat{k}$, $0 \leqslant i \leqslant n-1$, be the irreducible monic polynomial such that $P(\alpha)=0$. Set $P_Y:=\sum_{i=1}^{n} i g_i Y^{i-1}$.
         As $P$ is a separable polynomial, $\gcd (P,P_Y)=1$, and hence there exist $Q_1,Q_2 \in \hat{k}[Y]$ such that 
         \begin{equation}\label{p}
         	Q_1 P+Q_2 P_Y=1. 
         \end{equation}
         Let $\beta_1,\ldots,\beta_r \in k[x_1,\ldots,x_n] (=k^{[n]})$ denote the denominators of the coefficients of $P,Q_1,Q_2$ in $\hat{k}$. 
           Set $E_1:=k\left[ x_1,\ldots,x_n, \frac{1}{f\beta_1 \cdots \beta_r} \right]$. 
         Since $k$ is an infinite field, there exist a maximal ideal $\eta=(x_1-\lambda_1,\ldots,x_n-\lambda_n)$ of $k[x_1,\ldots,x_n]$ such that $f \beta_1 \cdots \beta_r \notin \eta$. 
         
%
    Let $\m_1$ be a maximal ideal of $E_1[\alpha]$ such that $\eta E_1[\alpha] \subseteq \m_1$. Therefore, from \eqref{p}, $k^{\prime}:= \frac{E_1[\alpha]}{\m_1}$ is a finite separable field extension of $k$. Now, from \eqref{b2}, we have 
    $$
    A \otimes_k k^{\prime} =(R \otimes_k k^{\prime})^{[2]},
    $$
    and hence the result follows by Case 1.
  \end{proof}

\section*{Acknowledgement}
The authors thank Professor Amartya K. Dutta for carefully going through the earlier draft and giving valuable suggestions. 
The second author acknowledges Department of Science and Technology (DST), India for their INDO-RUSS project \\(DST/INT/RUS/RSF/P-48/2021).

{\small{

}
}
\end{document}